\documentclass[preprint,review,12pt]{elsarticle}

\usepackage{graphicx}

\usepackage{amssymb,amsmath,amsthm}

\journal{Applied Mathematics and Computation}

\newtheorem{thm}{Theorem}[section]

\newtheorem{defn}{Definition}
\newtheorem{lem}[thm]{Lemma}
\newtheorem{prop}[thm]{Proposition}

\newtheorem{remark}[thm]{Remark}
\newtheorem{assume}[thm]{Assumption}

\newcommand{\R}{\mathbb{R}}

\newcommand{\obdom}{\Omega\setminus\Omega_h}
\newcommand{\obdomp}{\left(\Omega\setminus\Omega_h\right)}

\newcommand{\ddr}[1]{\frac{\partial #1}{\partial r}}
\newcommand{\ddz}[1]{\frac{\partial #1}{\partial z}}
\newcommand{\ddt}[1]{\frac{d#1}{dt}}

\begin{document}

\begin{frontmatter}


\title{Estimation of Near Surface Wind Speeds in Strongly Rotating Flows}  
\author[label1,label2]{Sean Crowell\corref{cor1}}
\ead[label1]{Sean.Crowell@noaa.gov}
\author[label1]{Luther White}
\ead[label2]{lwhite@ou.edu}
\author[label2]{Louis Wicker}
\ead[label3]{Louis.Wicker@noaa.gov}
\cortext[cor1]{Corresponding Author: scrowell@ou.edu, Voice: (405)325-1089, Fax: (405)325-1180}
\address[label1]{University of Oklahoma, Norman, OK 73019}
\address[label2]{NOAA National Severe Storms Laboratory, Norman, OK 73072}





\begin{abstract}
Modeling studies consistently demonstrate that the most violent winds in tornadic vortices occur in the lowest tens of meters above the surface.  These velocities are unobservable by radar platforms due to line of sight considerations.  In this work, a methodology is developed which utilizes parametric tangential velocity models derived from Doppler radar measurements, together with a tangential momentum and mass continuity constraint, to estimate the radial and vertical velocities in a steady axisymmetric frame.  The main result is that information from observations aloft can be extrapolated into the surface layer of the vortex.  The impact of the amount of information available to the retrieval is demonstrated through some numerical tests with pseudo-data.
\end{abstract}


\end{frontmatter}


\section{Introduction}

The strongest wind speeds in tornados are believed to occur a few tens of meters above the surface.  Due to line of sight limitations, radar platforms are typically unable to measure this portion of the atmosphere.  The relationship between the measurable flow aloft, and the unobservable (by radar) flow near the surface is complex (see for instance \citep{church1979}, \citep{fiedler86}, \citep{lewellen2000} and \citep{lewellen97} for different flow regimes).

The reviews \citep{lewellen93}, \citep{lewellen_nrc} and \citep{snow82}  discuss the dynamics of different sections of a tornado.  Snow (\citep{snow82}) describes the change in magnitude of the different wind components both in the vertical and radial directions, which is based on simulations in fluid dynamics models and in the Tornado Vortex Chamber \citep{church1977} at Purdue University.  A tornado with a positive vertical velocity along the central axis is called a ``single celled" vortex. The tangential velocity mean field increases as a function of height from ground level to a maximum, and then decreases again to the top of the vortex. Similarly, the tangential velocity increases as a function of the distance from the center of the vortex until it reaches a maximum, and then decreases to zero.  This behavior can be captured with empirical parametric models, such as those discussed in \citep{wood2011}.  Models of this type have also been used in observational studies such as \citep{wurman2005} to better understand measurements in the presence of noisy observations.

In this paper, we estimate the three components of the wind velocity near ground level from observations aloft.  The paper is divided into sections as follows.  In Section~\ref{sec:background}, we review the basic considerations regarding observations of atmospheric circulations by radar instruments and define the problem domain and relevant parameters of interest.  Section~\ref{sec:forwardmodel} introduces a method for estimating the vortex radial and vertical velocities, and Section~\ref{sec:surflayer} discusses the mathematical issues related to this method. The mathematical issues include positive aspects of estimating flow fields with these dynamics, as well as situations in which the dynamics are insufficient to estimate the flow on the entire domain.  Section~\ref{sec:modellimits} examines a few physical limitations of the approach.  In Section~\ref{sec:numexps}, we perform an identical twin experimental test of the method for a tornado-like vortex. We generate pseudo-observations with an assumed tangential velocity model and random errors. Then we estimate the flow using the same tangential velocity model.  This test is not meant to prove conclusively that the method will work with a real data set, but rather to show the theory in action.  

\begin{remark}
\emph{Many researchers in meteorology currently use variational techniques to estimate wind fields from radar velocity measurements.  These techniques are powerful, and are especially useful for dealing with noisy measurements.  They face the problems common to all optimal estimation techniques. Some of the difficulties are finding a unique global minimum and minimizer, and the tendency of least squares techniques to reduce the magnitude of smaller scale features.  Further, a minimizer of a set of weakly enforced constraints may not satisfy any of the constraints particularly well.  Boundary conditions for these types of methods are usually not chosen physically, but rather are allowed to be retrieved with the rest of the variables.  The authors are well acquainted with these techniques, and propose the techniques in this paper as a first step toward remedying some of these difficulties.  Most variational techniques utilize some sort of descent based minimization procedure, and the solutions provided by the method in this paper could be used as the ``first guess" which is required of all iterative schemes.}
\end{remark}

\section{Background}
\label{sec:background}

Assume that two radar instruments measure a given volume of air simultaneously.  The two horizontal components of the velocity can be recovered if the radar beams are approximately horizontal. In this case, the measurements contain very little information about the vertical component of velocity, i.e.~are orthogonal to the vertically pointing basis vector.  Take the flow to be in cylindrical coordinates, with the axis of the coordinate system aligned along the vertical axis of the vortex. Thus the recovered components are the tangential and radial components of the swirling flow.  

For the remainder of the this work, assume two sets of wind measurements, which have been converted to radial and tangential velocities for the vortex of interest, and averaged azimuthally to create an axisymmetric mean pair of velocities.  The spatial domain includes the vertical axis and the surface and measurements which are representable by a parametric model.  A family of parametric models for the tangential velocity is chosen which best approximate the qualitative features of the given data, then a particular parameter set is selected so that the tangential velocity model is optimal (in some sense).  This is done \emph{in advance} of seeking to estimate $u$ and $w$.

In the next section, the estimation of radial and vertical velocities in a layer near the surface, where the velocities are not observable, is considered.  The problem is posed on the domain $\Omega$, which is illustrated in Figure~\ref{fig:domain}.  The domain is decomposed into an observable region $\Omega_o$ and an unobservable region $\Omega_h$, separated by a horizontal line $z = h$. This line is referred to as the \emph{minimum observable height} (MOH) line.  The domain on which we interested in retrieving the flow is referred to as the \emph{surface layer}, which is the portion of the domain between the height $z = 0$ and $z = h_s$, where we will refer to $h_s$ as the \emph{surface layer height}. The parameter $h_s$ is chosen for the application of interest.  For example, if we are interested in surface damage, it might suffice to only examine the flow in the layer with $h_s = 1$ meter, whereas structural engineers might be interested in multistory buildings, and would necessarily use a larger value for this parameter.

\begin{figure}
\centering
\includegraphics[scale=0.5]{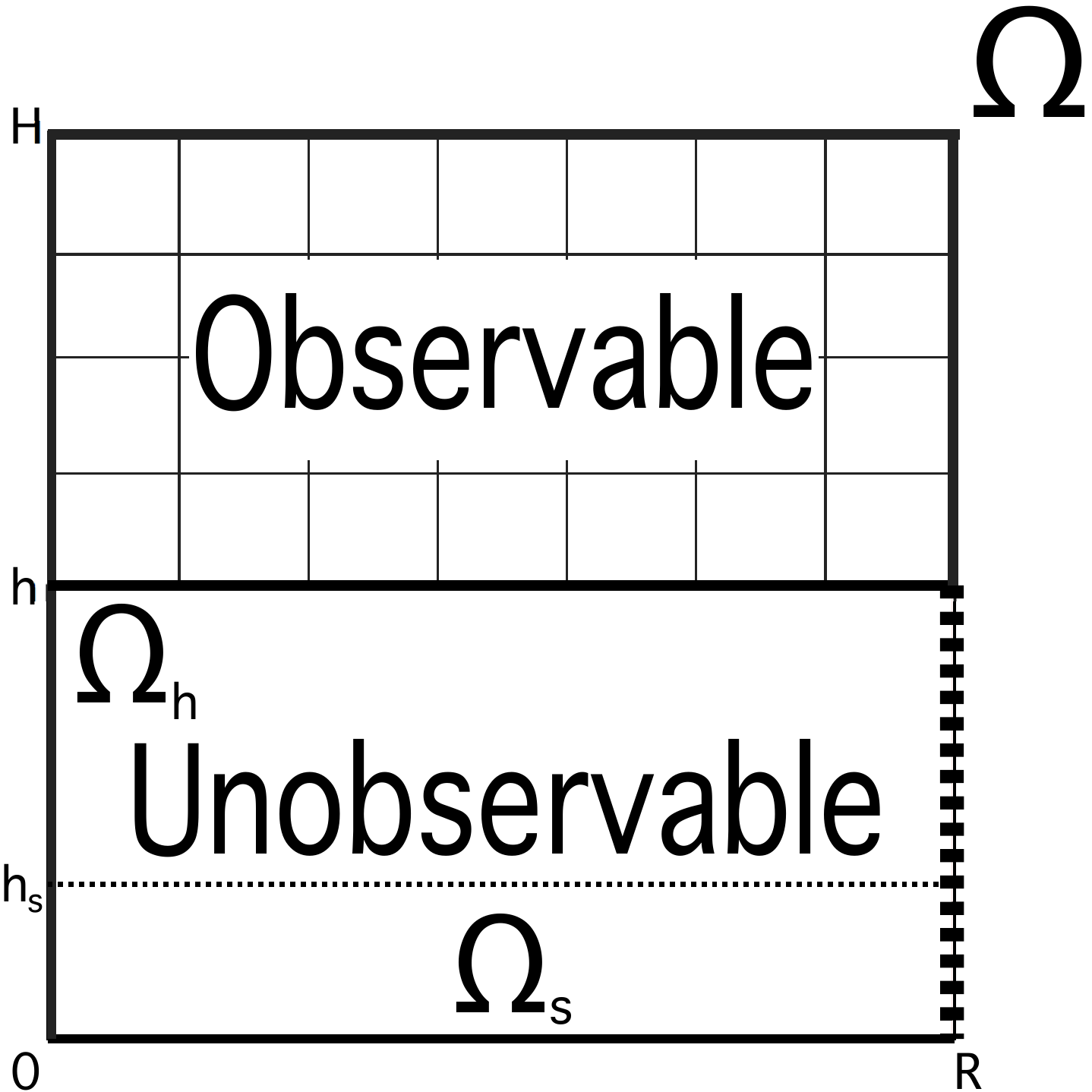}
\caption{Schematic of Problem Domain.  The outer radial boundary is dashed to represent the unknown boundary condition.}
\label{fig:domain}
\end{figure}

\section{Estimating $u$ and $w$}
\label{sec:forwardmodel}

Assume that the vortex is approximately steady and axisymmetric, and that $v(r,z)$ captures the essence of the tangential velocity present in the observations.  Consider the steady, axisymmetric Navier-Stokes equations of motion, given by
\begin{align}
u\ddr{u} + w\ddz{u} - \frac{v^2}{r} &=  -\frac{1}{\rho} \ddr{p} +\nu\left\{\frac{\partial}{\partial r}\left[\frac{1}{r}\ddr{(ru)}\right]+\frac{\partial^2 u}{\partial z^2}\right\} \label{eq:uasstd} \\
	u\ddr{v} + w\ddz{v}+ \frac{uv}{r}  &= \quad\quad\quad\quad\nu\left\{\frac{\partial}{\partial r}\left[\frac{1}{r}\ddr{(rv)}\right]+\frac{\partial^2 v}{\partial z^2}\right\} \label{eq:vasstd} \\
	u\ddr{w}+w\ddz{w}\phantom{+\frac{1}{r}}\quad&= -\frac{1}{\rho}\ddz{p} + \nu\left\{\frac{1}{r}\frac{\partial}{\partial r}\left[r\ddr{w}\right]+\frac{\partial^2 w}{\partial z^2}\right\} \label{eq:wasstd}
\end{align}
where $u$, $v$ and $w$ represent the three components of the velocity vector in cylindrical coordinates, $\rho$ the density, $p$ the pressure, and $\nu$ the fluid viscosity.  Further, assume that the fluid is incompressible, and so mass conservation takes the form
\begin{align}
\frac{1}{r} \ddr{(ru)} + \ddz{w} = 0 \label{eq:cont}
\end{align}
If $v(r,z)$ were a component of a true solution to \eqref{eq:uasstd}-\eqref{eq:wasstd}, then this system would still have a unique solution if $v$ exactly modeled error free data, and if these equations exactly hold for real atmospheric vortices.  Realistically, observational and model errors lead us to conclude that enforcing only a subset of these dynamics may help to avoid an overdetermined problem.  

Introduce the vertical vorticity $\zeta(r,z) = \frac{1}{r}\ddr{(rv)}$ and the radial vorticity $\eta(r,z) = -\ddz{v}$.  With these substitutions, \eqref{eq:vasstd} can be rewritten as
\begin{align}
\zeta(r,z) u(r,z) -\eta(r,z) w(r,z) = \nu \left[ \ddr{\zeta}(r,z)-\ddz{\eta}(r,z)\right] \label{eq:vmom} 
\end{align}
which is an algebraic relation between $u$ and $w$, once $v$ has been selected.  Next, introduce a streamfunction $\Psi$, defined by
\begin{align*}
	\frac{1}{r} \ddz{\Psi}(r,z) &= u(r,z) \\
	-\frac{1}{r} \ddr{\Psi}(r,z) &= w(r,z)
\end{align*}
in cylindrical coordinates, so that $\Psi$ satisfies \eqref{eq:cont} automatically.  The tangential momentum equation \eqref{eq:vmom} becomes
\begin{align}
\zeta(r,z) \ddz{\Psi}(r,z) + \eta(r,z) \ddr{\Psi}(r,z) = \nu r \left[ \ddr{\zeta}(r,z)-\ddz{\eta}(r,z)\right] \label{eq:strfcnpde}
\end{align}
This is a hyperbolic boundary value problem on $\Omega_h$.  The boundary conditions at the surface and vertical axis should yield vortical flows similar to actual atmospheric vortices.  By choosing zero Dirichlet boundary conditions for $\Psi$ on the lower and axial boundaries, mass is conserved.  The boundary condition for $u$ along $z = h$ is provided by the measurements, and the boundary condition for $w$ is taken to be the result of solving \eqref{eq:vmom} for $w$ and substituting in the condition for $u$.  Once $w(r,z)$ is known along the MOH line, $\Psi$ is recovered using
\begin{align}
	\Psi(r,h) &= -\int_0^r s w(s,h) ds.
\end{align}
The outer radial boundary is left unconstrained for the moment.

Equation \eqref{eq:strfcnpde} is quasilinear with associated characteristic equations \cite{evans1998}:
\begin{align}
	\ddt{r} &= \eta(r,z) \label{eq:rchar}\\
	\ddt{z} &= \zeta(r,z) \label{eq:zchar}\\
	\ddt{\Psi} &= \nu r \left[ \ddr{\zeta}(r,z)-\ddz{\eta}(r,z)\right] \label{eq:pchar},
\end{align}
where $t$ is the characteristic variable for the position along the characteristic curve given by $(r(t),z(t))$.  To seek solutions these ordinary differential equations must be supplemented with initial conditions.  Let $s$ denote the characteristic variable which distinguishes between different characteristic curves, by parameterizing the initial values for $r$, $z$, and $\Psi$, and define
\begin{align}
	r(0,s) &= s \label{eq:rinit}\\
	z(0,s) &= h \label{eq:zinit}\\
	\Psi(0,s) &= -\int_0^s s w(s,h) ds \label{eq:pinit}
\end{align}
This choice of initial conditions means that the equations are initialized with values on the upper boundary of $\Omega_h$ and allow the dynamics to propagate the information contained on them down into the domain.

\begin{remark}\emph{Assuming that $v \in C^k(\Omega_h)$ for some $k \geq 2$, classical results from the theory of ordinary differential equations (for example, those in \cite{dieudonne1960}) provide existence and uniqueness of solutions to these initial value problems, and smoothness with respect to the initial conditions.  This implies that if a point $(r,z)$ lies on a characteristic curve that intersects the upper boundary of $\Omega_h$, there is a classical solution $\Psi$ defined at $(r,z)$ that satisfies \eqref{eq:pchar} and \eqref{eq:pinit}.}  \end{remark}

\begin{remark}\emph{The fluid viscosity $\nu$ is an important physical constant for the purposes of time dependent model simulation.  Since the flow is stationary, $\nu$ has a small impact on the results with this method. Where it makes calculations simpler, $\nu$ will be set to 0.}\end{remark}

\noindent In order to simplify the discussion, we introduce the following notation.

\begin{defn} For a point $(r_o,z_o) \in \Omega_h$, define
	\begin{itemize}
\item[(1)] $c(\cdot,r_o,z_o):\R\rightarrow\Omega$: the solution mapping of the dynamical system \eqref{eq:rchar}-\eqref{eq:zchar} with initial condition $(r_o,z_o)$.
\item[(2)] $C(r_o,z_o) = c(\R,r_o,z_o)$: the set of all points $(r,z)$ which can be attained by integrating  \eqref{eq:rchar}-\eqref{eq:zchar} (either forwards or backwards) starting from $(r_o,z_o)$.
\item[(3)] $K_h = \{(r_o,z_o)\in\Omega_h | (r,h) \notin C(r_o,z_o)\phantom{1} \forall r \in [0,R]\}$.
\end{itemize}
\end{defn}
\noindent $C(r,z)$ is referred to as the \emph{characteristic curve containing} $(r,z)$. The set $K_h$ is referred to as the \emph{information void} for the problem, because the dynamics do not carry information from aloft to these points.

\section{Surface Layer Wind Velocities}
\label{section:surflayer}

\noindent Assume that $v(r,z) = \phi(r)\psi(z)$, where   
\begin{assume}\label{ass:singlemax}
	\begin{itemize}
		\item[(1)] $\phi$ and $\psi$ both are $k$ times continuously differentiable $(k \geq 2)$.
		\item[(2)] \emph{ (no-slip condition) } $\phi(0) = \psi(0) = 0$. 
		\item[(3)] $\phi > 0$ on $(0,R)$ and $\psi > 0$ on $(0,H)$.
		\item[(4)] $\frac{d\phi}{dr}(r_o) + \frac{1}{r_o}\phi(r_o) = 0$ and $\frac{1}{r}\frac{d(r\phi)}{dr} \neq 0$ for $r \neq r_o$.
		\item[(5)] $\frac{d\psi}{dz}(z_o) = 0$ and $\frac{d\psi}{dz} \neq 0$ for $z \neq z_o$.
	\end{itemize}
\end{assume}
This assumption allows a more thorough analysis, and \citep{wood2011} has demonstrated the utility of such models for data analysis. A schematic streamfunction of a vortex embodied in these assumptions is shown in Figure~\ref{fig:vsinglemax}.
\begin{figure}
	\centering
	\includegraphics[scale=0.5]{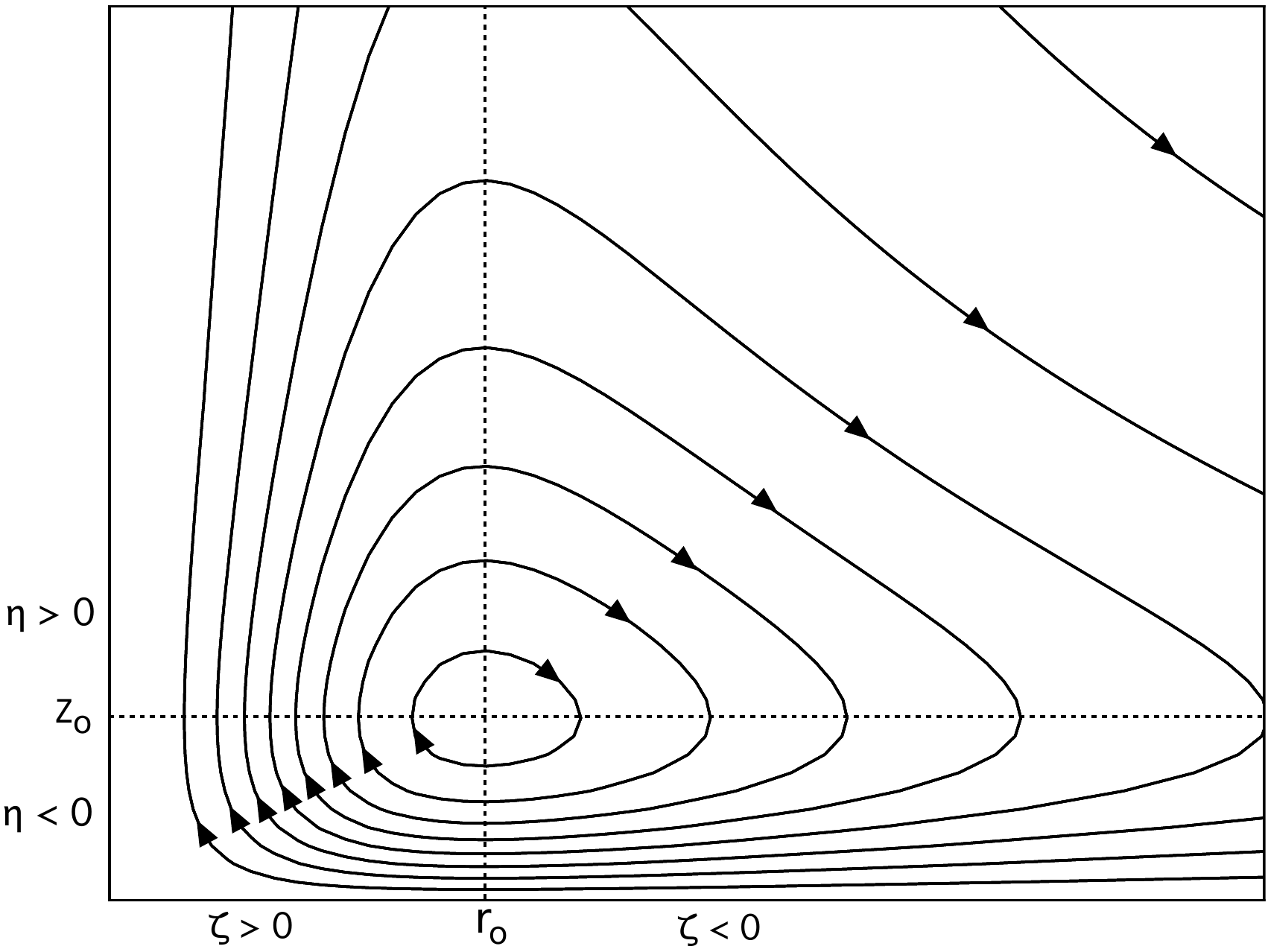}
	\caption{Schematic of Characteristic Curves under Assumption~\ref{ass:singlemax}, where the relative maximum of $\Gamma$ occurs at $(r_o,z_o)$.}
	\label{fig:vsinglemax}
\end{figure}
\\ \\
\noindent The following result says that Assumption~\ref{ass:singlemax} always yields a nontrivial surface layer in which we can retrieve the flow:

\begin{thm} \label{thm:nontrivialregionabovesurface} If Assumption~\ref{ass:singlemax} holds, then there is an $h_o$ such that if $z < h_o$, $C(r,z) \cap \obdomp \neq \emptyset$. \end{thm}

\noindent The next four results form the basis of the proof of Theorem~\ref{thm:nontrivialregionabovesurface}.

\begin{lem}\label{thm:charcurves} The characteristic curves $C(r,z)$ are the level curves of $\Gamma$, where $\Gamma = rv$ is the circulation on circles parallel to the horizontal plane, centered on the vertical axis.
\end{lem}
\begin{proof}
When $\eta \neq 0$, we can write the solution curves as $(r,z(r))$ by considering
\begin{align}
\frac{dz}{dr} = \frac{\zeta}{\eta} \label{symeqnz}
\end{align}
and when $\zeta \neq 0$ as $(r(z),z)$ from
\begin{align}
\frac{dr}{dz} = \frac{\eta}{\zeta}. \label{symeqnr}
\end{align}
Note that
\begin{align*}
\frac{\zeta}{\eta} = \frac{(rv)_r}{-rv_z} = -\frac{\Gamma_r}{\Gamma_z}
\end{align*}
which implies that the characteristic curves are everywhere tangent to the level curves of $\Gamma$.  Hence, viewed in the plane, these two collections of curves are the same.
\end{proof}
\begin{remark}\emph{This result specifies the characteristic curves in terms of our tangential velocity model, which is estimated \emph{a priori} utilizing a least squares (or some other) data mismatch criterion.  It also gives a criteria by which to avoid the solution $\Psi$ being multiply defined, which can occur when using method of characteristics.  To avoid this behavior, choose $v$ to be appropriately smooth.}\end{remark}

\noindent The next result states that when the maximum tangential velocity is in the observable region, then the flow is retrievable over all of $\Omega$ using characteristics.
\begin{lem} \label{thm:zcgreaterthanh} If $h < z_o$, then $K_h = \emptyset$.
\end{lem}
\begin{proof}
There are two cases.  For $r \leq r_o$, $\eta<0$ and $\zeta \geq 0$.  Hence if $C(r,z)$ is traversed in the positive $t$ direction, the curve must eventually cross $z = h$, since $C(r,z)$ cannot intersect the vertical axis.  For $r > r_o$, $\eta <0$ and $\zeta < 0$.  Since $C(r,z)$ cannot intersect the horizontal axis, there must be a $z_1$ such that $(r_o,z_1) \in C(r,z)$, and now apply the argument from the first case, using $(r_o,z_1)$ as our initial point.  Hence, for any $(r,z) \in \Omega_h$, $C(r,z) \cap \obdomp \neq \emptyset$, which implies $K_h = \emptyset$.
\end{proof}

\begin{lem}\label{thm:singlemaxclosedcurves}
Suppose that $0 < z_1 < H$ satisfies $C(r_o, z_1) \subset \Omega^o$.  Then $C(r_o,z_1)$ is a closed curve.
\end{lem}
\begin{proof}
Let $z_1>z_o$.  Since $\Gamma$ has only a single relative maximum, $\frac{\partial \Gamma}{\partial r} > 0$ for $r < r_o$ and $\frac{\partial \Gamma}{\partial r} < 0$ for $r > r_o$, and similarly for the vertical gradient of $\Gamma$.  Consider the characteristic curve which passes through $(r_o,z_1)$, and first traverse in the positive $t$ direction.  Since $\Gamma_z(r_o,z_1) < 0$, $\eta > 0$, and so the characteristic curve moves to the right.  For $r > r_o$ and $z>z_o$, $\eta >0$ and $\zeta < 0$, and so the characteristic curve moves to the right and down.  Since $C(r_0,z_1) \subset \Omega^o$, there must be an $r_1$ with $r_o <r_1<R$ such that $(r_1,z_o) \in C(r_o,z_1)$, else $C(r_o,z_1)$ would cross the line $r = R$.  Similarly, since for $r > r_o$ and $z < z_o$, $\eta <0$ and $\zeta < 0$, there must be a $0<z_2<z_0$ such that $(r_o,z_2)\in C(r_o,z_1)$.  Otherwise $C(r_o,z_1)$ would intersect the lower axis $z = 0$, which would contradict Corollary~\ref{thm:charaxes}.  Thus $C(r_o,z_1)$ intersects the line $r = r_o$ at $(r_o,z_2)$.  By traversing $C(r_o,z_1)$ in the negative $t$ direction starting from $(r_o,z_1)$, and using similar arguments, there is a $0<z_3<z_o$ such that $(r_o,z_3)\in C(r_o,z_1)$.

Suppose $z_2<z_3$.  Then there is a $z^*$ with $z_2<z^*<z_3$, and since $\Gamma_z > 0$, we must have that 
\begin{align}
\Gamma(r_o,z_2)<\Gamma(r_o,z^*)<\Gamma(r_o,z_3).
\end{align}
But this is a contradiction, since $\Gamma(r_o,z_2) = \Gamma(r_o,z_3)$.

Let $t_2$ such that $c(t_2,r_o,z_1) = (r_o,z_2)$ and $t_3$ such that $c(-t_3,r_o,z_1) = (r_o,z_3)$.  Then $c(t_2+t_3,r_o,z_1) = (r_o,z_1)$ and $C(r_o,z_1)$ is closed.

If $z_1 < z_o$, a symmetric argument shows that $C(r_o,z_1)$ is closed.
\end{proof}

\begin{lem}\label{thm:characterizationofk} Suppose that $z_o<h$.  Then one and only one of the following statements holds:
\begin{itemize}
\item[(1)]  $C(r_o,h)$ is a closed curve, and $K_h$ is the interior of the region enclosed by $C(r_o,h)$.
\item[(2)] $C(r_o,h)$ intersects the outer radial boundary at $(R,z_1)$ and $(R,z_2)$, and $K_h$ is the interior of the region enclosed by $C(r_o,h)$ and the segment $\{(R,z):z_1\leq z \leq z_2\}$
\end{itemize}
\end{lem}
\begin{proof}
First, if $C(r_o,h)$ is not a closed curve, then if $C(r_o,h)$ is traversed in the negative $t$ direction, it must cross the line $z = z_o$, and then the line $r = r_o$, because $C(r_o,h)$ cannot intersect the axes.  This implies that there is a $t$ such that $c(-t,r_o,h) = (r^*,z^*)$ with $r^* > r_o$ and $z^*<z_o$.  If $C(r_o,h)$ were to cross the line $z = z_o$ again, then the signs of the vorticities would force $C(r_o,h)$ to intersect $r = r_o$, and at the point $(r_o,h)$ by the argument in Proposition~\ref{thm:singlemaxclosedcurves}.  Similarly, if $C(r_o,h)$ is traversed in the positive $t$ direction, $C(r_o,h)$ cannot cross the line $z = z_o$, or else $C(r_o,h)$ would be a closed curve.  Thus, either $C(r_o,h)$ is a closed curve, or $C(r_o,h)$ intersects the outer radial boundary at two distinct points $(R,z_1)$ and $(R,z_2)$, where $z_1 < z_o < z_2$.  In either case, denote the set enclosed by $C(r_o,h)$ (and possibly $\{R\} \times [z_1,z_2]$) by $K_o$.

If $(r,z) \in \Omega_h \setminus K_o$, proceed as before by traversing $C(r,z)$ either in the positive ($r < r_o$ or $z < z_o$) or negative ($r > r_o$ and $z > z_o$) $t$ direction.  Note that $c(t,r,z) \notin K_o$ for all $t \in \R$ because $\partial K_o = C(r_o,h)$ (possibly plus the outer boundary), and characteristic curves may not intersect. Since $c(t,r,z)$ also cannot intersect the axes, there must be a $t$ such that $c(t,r,z) \in \obdomp$.  Thus $K \subset K_o$. 

If $(r_1,z_1) \in K_o$, then $\Gamma(r_1,z_1) > \Gamma(r,z)$ for all $(r,z) \in \obdom$.  Thus $C(r_1,z_1) \cap \obdomp = \emptyset$ and so $(r_1,z_1) \in K$.  Hence $K_o \subset K$, and $K = K_o$.
\end{proof}

\subsection*{Proof of Theorem~\ref{thm:nontrivialregionabovesurface}:}
\begin{proof}
If $z_o \geq h$, simply take $h_o = h$ by Lemma~\ref{thm:zcgreaterthanh}.  Suppose $z_o < h$.d Since $C(r_o,h) \subset \Omega$ is either closed or intersects the outer radial boundary, it is also compact.  Hence the map $(r,z) \mapsto z$ has a minimizer at some point $h_o$.  Thus, if $z < h_o$, $(r,z) \notin K_h$, and so $C(r,z) \cap \obdomp \neq \emptyset$ by Proposition~\ref{thm:characterizationofk}.
\end{proof}

\begin{remark}\emph{The height  $z = h_o$  can be referred to as the \emph{minimum unreachable height}, since for values of $z<h_o$, the solution is reachable via characteristic curves.  Under Assumption~\ref{ass:singlemax}, the proof of Theorem~\ref{thm:nontrivialregionabovesurface} implies that $h_o$ is the smallest solution of $\psi(z) = \psi(h)$.} \end{remark}

\begin{remark}\emph{A similar result holds for the map $(r,z) \mapsto r$, implying the existence of a ``minimum unreachable radius", though this is not directly relevant to the problem initially posed.}\end{remark}

\section{Model Limitations}
\label{sec:modellimits}

\subsection{Boundary Conditions}

\noindent The following corollary follows immediately from Assumption~\ref{ass:singlemax} and Lemma~\ref{thm:charcurves}
\begin{prop} \label{thm:charaxes} If Assumption~\ref{ass:singlemax} holds, then no characteristic curve may intersect the lines $r = 0$ and $z = 0$.
\end{prop}
\begin{proof} Note that $\Gamma(0,z) = \Gamma(r,0) = 0$.  Since $\Gamma > 0$ on the interior of $\Omega_h$, no level curve of $\Gamma$ intersections the boundaries, and hence no characteristic curves intersect the axes. \end{proof}  

\begin{remark}\emph{Proposition~\ref{thm:charaxes} implies that the choice of boundary conditions along the surface and the vertical axis do not affect the flow on the interior of the domain, so long as $v$ vanishes on these axes.  This is a consequence of the choice of dynamical constraints, and removes a physical degree of freedom from the problem, since in real vortices, surface roughness effects can propagate into the domain.The literature contains multiple discussions (e.g. \cite{fiedler86}, \cite{lewellen93}) of what boundary conditions are most realistic, and generate physically realistic vortices.  It is intuitively clear that the radial and vertical velocities will depend on the their behavior at the surface and along the central vertical axis, but this is not captured by the dynamics we are choosing to constrain the solution.}
\end{remark}

\subsection{Multiple MOH Intersections}
Another difficulty is the possibility of characteristic curves intersecting the MOH line multiple times.  In this case, the boundary data on the MOH line may not be compatible with the dynamics.  For real data, this will almost certainly not be the case due to noise and the error introduced by the tangential model $v$.  This situation is reminiscent of the data assimilation problem that is usually tackled using least squares minimization of an objective functional that penalizes disagreement between model prediction and observation relative to the uncertainty present in each.  More information about this topic is found in \cite{lakshmivarahan2006}.  This problem will be addressed in a future work.

\subsection{Velocities above the Minimum Unreachable Height}
The results in Section~\ref{section:surflayer} point to potential difficulties when $h_s > h_o$, the minimum unreachable height guaranteed by Theorem~\ref{thm:nontrivialregionabovesurface}.  Clearly, there are portions of this set that are reachable by characteristics, namely those characteristic curves that pass through to the surface layer below $h_o$.  The rest of $\Omega_h$ is precisely $K_h$, which we have called the information void.

\section{Numerical Experiments}
\label{sec:numexps}

In this section, a simple test of the theory developed in Section~\ref{section:surflayer} is demonstrated.  This experiment is an identical twin, since the same functional form is used to generate the observations as the one used to select $v$ and estimate $u$ and $w$.  Initial tests showed that dependence on the viscosity $\nu$ was small.  With this in mind, assume $\nu = 0$, which simplifies the numerics from solution of a linear ordinary differential equation for $\Phi$ to solving the equation $\Phi(r,z) = \Phi(r_o,h)$. This equation can be approximately solved to any specified degree of accuracy using a simple bisection method.

\subsection{Generation of Pseudo-observations}

As a first experiment, a collection of pseudo-observations is generated that emulates a single time of model output from Davies-Jones' axisymmetric model, described in \cite{dj2008}.  At the time of interest, the tangential velocity near the surface exhibits a single maximum.  The radial velocity is negative beneath this maximum, which is typical of a swirling flow (\cite{rotunno79}) with a no-slip condition on the tangential velocity, and represents air being drawn into the vortex.  Finally, the vertical velocity is relatively large and positive along the axis adjacent to the tangential maximum, which is also typical of these types of flows.  The tangential velocity $v(r,z)$ is modeled using a product of functions of the form
\begin{align}
	\phi_{ww}(x,n,x_c) = \frac{n{x_c}^{n-1}x}{(n-1){x_c}^n+x^n}.
\end{align}
The function $\phi$ has a smooth maximum at $(x_c,1)$, and increases approximately linearly on $(0,x_c)$, and decays like $x^{n-1}$ as $x\rightarrow\infty$.  Assume 
\begin{align*}
v(r,z) = v_c\phi_{ww}(r,n_r,r_c)\phi_{ww}(z,n_z,z_c).
\end{align*}
This function satisfies Assumption~\ref{ass:singlemax}, and so all of the theory in Section~\ref{sec:surflayer} is valid for this choice of model.  The velocity pseudo-obs used are depicted in Figure~\ref{fig:pseudobs}, with the tangential velocity depicted as contours, and the radial and vertical velocities depicted as a single vector in the $r-z$ plane.  To simulate the effects of measurement error, a set $\Sigma$ of independent realizations of a normal distribution with standard deviation equivalent to 1 $ms^{-1}$ was added to the tangential velocity values on a discrete spatial grid. Three separate experiments were run to simulate random errors with deviation as large as 3 $ms^{-1}$.

\begin{figure}[h]
\centering
\includegraphics[scale=0.8]{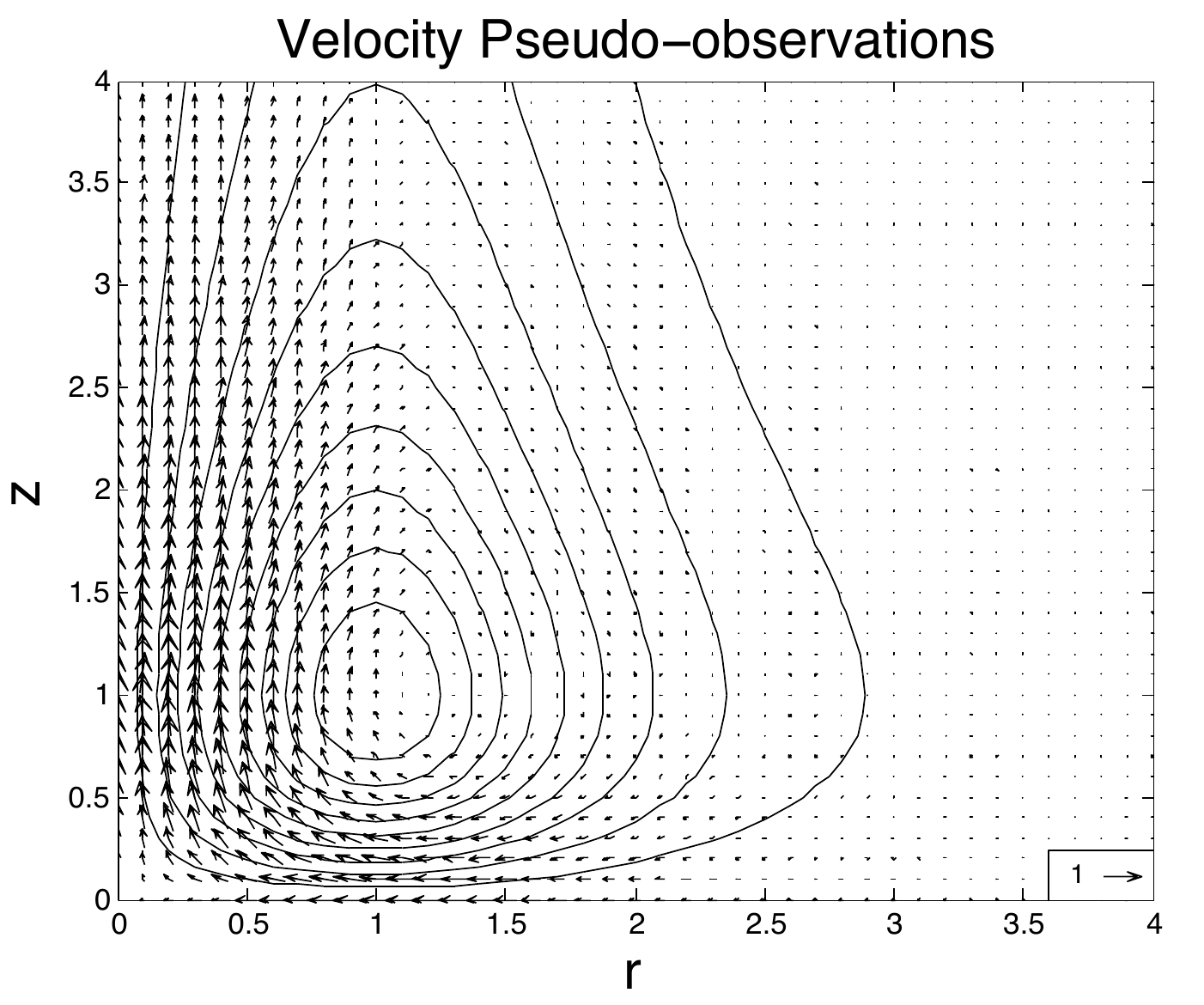}
\caption{The ``truth" velocity fields, where the contours are of tangential velocity with increment 0.1 from 0 to 1, and the vectors indicate the radial and vertical velocities, with a reference vector of length 1 shown in the bottom right corner.  Note the strength of the inflow near the surface, and of the updraft near the vertical axis.}
\label{fig:pseudobs}
\end{figure}

\subsection{Impact of MOH on Surface Layer Thickness}
As a demonstration of Theorem~\ref{thm:nontrivialregionabovesurface}, the streamfunction for a fixed initial condition was computed for MOH values of 1.5, 2.5 and 3.5.  The resulting surface layer streamfunction is plotted in Figure~\ref{fig:comp_hs_comparison}.  Note that as more of the vortex is observable, more is retrieved below the MOH line.  Also note that even in the case with the least information (h = 3.5), there is a retrieved surface layer of nontrivial thickness.

\begin{figure}[h]
\centering
\includegraphics[scale=0.6]{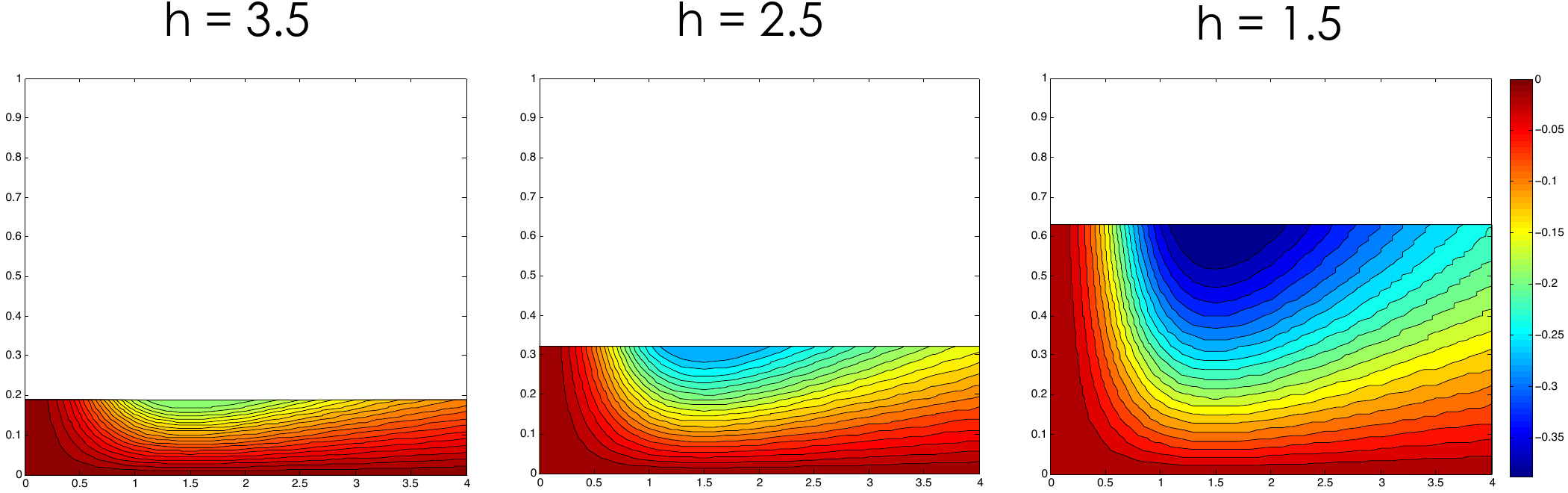}
\caption{Comparison of Retrieved Surface Layers for Three Different MOH Values}
\label{fig:comp_hs_comparison}
\end{figure}

\section{Discussion}

A methodology was introduced for extrapolating observations of wind velocities downward toward the surface.  For the dynamics chosen to constrain the flow, the information contained in observations aloft propagates along curves that coincide with the level curves of $\Gamma = rv$, which is estimated from observations in advance of the problem discussed here, and hence is known \emph{a priori}.  With more assumptions about our tangential model, the location and size of the information void $K_h$ are exactly known for a specific value of $h$.  An important result is that there is always a nontrivial height $h_o$, below which, everything can be retrieved using the characteristic framework.

As a first test of the method, a set of pseudo-observations was created that approximate the behavior of simulated tornadoes. The flow was then estimated assuming different amounts of knowledge, embodied in the parameter $h$, and different levels of random error.  These experiments yield mixed results, because it always occurs that the true value of $u^+$, $w^+$, or $|\vec{v}|_{\max}$ falls inside the spread of results, but not always with the correct frequency (relative to agreement with the tangential velocity observations).  It is clear from the theory that the functional form chosen for $v(r,z)$  has strong impacts on what can be retrieved, as well as the quality of what is retrieved.  

The authors assert that knowledge of the shortcomings of a particular method is valuable information; hence, the analysis of and focus on information voids.  In a more standard variational technique, these voids would not appear, because the various smoothness terms would ensure that a smooth solution is defined everywhere.  However, \emph{the solution in those regions is no more physically relevant than any other solution}, since it is completely determined by the terms that are introduced for numerical stability.  A natural next step is the inclusion of another dynamical constraint, such as the balance equation for azimuthal vorticity.  This extra constraint could yield a method to estimate a meaningful solution in the information voids, without relying on unphysical smoothing.

\bibliographystyle{model2-names}
\bibliography{tornado}

\begin{thebibliography}{16}
\expandafter\ifx\csname natexlab\endcsname\relax\def\natexlab#1{#1}\fi
\expandafter\ifx\csname url\endcsname\relax
  \def\url#1{\texttt{#1}}\fi
\expandafter\ifx\csname urlprefix\endcsname\relax\def\urlprefix{URL }\fi
\providecommand{\eprint}[2][]{\url{#2}}
\providecommand{\bibinfo}[2]{#2}
\ifx\xfnm\relax \def\xfnm[#1]{\unskip,\space#1}\fi
\bibitem[{Church et~al.(1977)Church, Snow and Agee}]{church1977}
\bibinfo{author}{Church, C.}, \bibinfo{author}{Snow, J.},
  \bibinfo{author}{Agee, E.}, \bibinfo{year}{1977}.
\newblock \bibinfo{title}{{Tornado Vortex Simulation at Purdue University}}.
\newblock \bibinfo{journal}{Bull. Am. Met. Soc.} \bibinfo{volume}{58},
  \bibinfo{pages}{900--908}.
\bibitem[{Church et~al.(1979)Church, Snow, Baker and Agee}]{church1979}
\bibinfo{author}{Church, C.}, \bibinfo{author}{Snow, J.},
  \bibinfo{author}{Baker, G.}, \bibinfo{author}{Agee, E.},
  \bibinfo{year}{1979}.
\newblock \bibinfo{title}{{Characteristics of Tornado-Like Vortices as a
  Function of Swirl Ratio: A Laboratory Investigation}}.
\newblock \bibinfo{journal}{J Atmos Sci} \bibinfo{volume}{36},
  \bibinfo{pages}{1755--1776}.
\bibitem[{Crowell(2011)}]{crowelldissertation}
\bibinfo{author}{Crowell, S.}, \bibinfo{year}{2011}.
\newblock \bibinfo{title}{{Estimation of Near Surface Tornadic Wind Speeds}}.
\newblock Ph.D. thesis. University of Oklahoma.
\bibitem[{Davies-Jones(2008)}]{dj2008}
\bibinfo{author}{Davies-Jones, R.}, \bibinfo{year}{2008}.
\newblock \bibinfo{title}{Can a descending rain curtain in a supercell
  instigate tornadogenesis barotropically?}
\newblock \bibinfo{journal}{J. Atmos Sci} \bibinfo{volume}{65},
  \bibinfo{pages}{2469--2497}.
\bibitem[{Dieudonne(1960)}]{dieudonne1960}
\bibinfo{author}{Dieudonne, J.}, \bibinfo{year}{1960}.
\newblock \bibinfo{title}{Foundations of Modern Analysis}.
\newblock Number~\bibinfo{number}{10} in \bibinfo{series}{Pure and Applied
  Mathematics}, \bibinfo{publisher}{Academic Press}, \bibinfo{address}{New
  York}.
\bibitem[{Evans(1998)}]{evans1998}
\bibinfo{author}{Evans, L.}, \bibinfo{year}{1998}.
\newblock \bibinfo{title}{Partial Differential Equations}.
  volume~\bibinfo{volume}{19} of \textit{\bibinfo{series}{Graduate Studies in
  Mathematics}}.
\newblock \bibinfo{publisher}{American Mathematical Society}.
\bibitem[{Fiedler and Rotunno(1986)}]{fiedler86}
\bibinfo{author}{Fiedler, B.}, \bibinfo{author}{Rotunno, R.},
  \bibinfo{year}{1986}.
\newblock \bibinfo{title}{A theory for the maximum windspeeds in tornado-like
  vortices}.
\newblock \bibinfo{journal}{J. Atmos Sci} \bibinfo{volume}{43},
  \bibinfo{pages}{2328--2340}.
\bibitem[{Lakshmivarahan et~al.(2006)Lakshmivarahan, Lewis and
  Dhall}]{lakshmivarahan2006}
\bibinfo{author}{Lakshmivarahan, S.}, \bibinfo{author}{Lewis, J.},
  \bibinfo{author}{Dhall, S.}, \bibinfo{year}{2006}.
\newblock \bibinfo{title}{{Dynamic Data Assimilation: A least squares
  approach}}.
\newblock \bibinfo{publisher}{Cambridge University Press}.
\bibitem[{Lewellen et~al.(2000)Lewellen, Lewellen and Xia}]{lewellen2000}
\bibinfo{author}{Lewellen, D.}, \bibinfo{author}{Lewellen, W.},
  \bibinfo{author}{Xia, J.}, \bibinfo{year}{2000}.
\newblock \bibinfo{title}{The influence of local swirl ratio on tornado
  intensification near the surface}.
\newblock \bibinfo{journal}{J Atmos Sci} \bibinfo{volume}{57},
  \bibinfo{pages}{527--544}.
\bibitem[{Lewellen(1993)}]{lewellen93}
\bibinfo{author}{Lewellen, W.}, \bibinfo{year}{1993}.
\newblock \bibinfo{title}{Tornado vortex theory}, in: \bibinfo{editor}{Church,
  C.}, \bibinfo{editor}{Burgess, D.}, \bibinfo{editor}{Doswell, C.},
  \bibinfo{editor}{Davies-Jones, R.} (Eds.), \bibinfo{booktitle}{{The Tornado:
  Its Structure, Dynamics, Prediction, and Hazards}.}.
  \bibinfo{publisher}{American Geophysical Union}. number~\bibinfo{number}{79}
  in \bibinfo{series}{Geophysical Monograph}, pp. \bibinfo{pages}{19--39}.
\bibitem[{Lewellen and Lewellen(1997)}]{lewellen97}
\bibinfo{author}{Lewellen, W.}, \bibinfo{author}{Lewellen, D.},
  \bibinfo{year}{1997}.
\newblock \bibinfo{title}{Large-eddy simulation of a tornado's interaction with
  the surface}.
\newblock \bibinfo{journal}{J Atmos Sci} \bibinfo{volume}{54},
  \bibinfo{pages}{581--605}.
\bibitem[{Lewellen and Sheng(1980)}]{lewellen_nrc}
\bibinfo{author}{Lewellen, W.}, \bibinfo{author}{Sheng, Y.},
  \bibinfo{year}{1980}.
\newblock \bibinfo{title}{{Modeling Tornado Dynamics}}.
\newblock \bibinfo{type}{Contract Rpt. NUREG/CR-1585.} \bibinfo{number}{ARAP
  Report 421}. U.S. Nucl. Regul. Comm.
\bibitem[{Rotunno(1979)}]{rotunno79}
\bibinfo{author}{Rotunno, R.}, \bibinfo{year}{1979}.
\newblock \bibinfo{title}{{A Study in Tornado-Like Vortex Dynamics}}.
\newblock \bibinfo{journal}{J Atmos Sci} \bibinfo{volume}{36},
  \bibinfo{pages}{140--155}.
\bibitem[{Snow(1982)}]{snow82}
\bibinfo{author}{Snow, J.}, \bibinfo{year}{1982}.
\newblock \bibinfo{title}{{A Review of Recent Advances in Tornado Vortex
  Dynamics}}.
\newblock \bibinfo{journal}{Reviews of Geophysics and Space Physics}
  \bibinfo{volume}{20}, \bibinfo{pages}{953--964}.
\bibitem[{Wood and White(2011)}]{wood2011}
\bibinfo{author}{Wood, V.}, \bibinfo{author}{White, L.}, \bibinfo{year}{2011}.
\newblock \bibinfo{title}{{A new parametric model of vortex tangential
  wind-profile: Development, testing and verification}}.
\newblock \bibinfo{journal}{J. Atmos Sci} \bibinfo{volume}{68},
  \bibinfo{pages}{990--1006}.
\bibitem[{Wurman and Alexander(2005)}]{wurman2005}
\bibinfo{author}{Wurman, J.}, \bibinfo{author}{Alexander, C.},
  \bibinfo{year}{2005}.
\newblock \bibinfo{title}{{The 30 May 1998 Spencer, South Dakota, Storm. Part
  II: Comparison of Observed Damage and Radar-Derived Winds in the Tornadoes}}.
\newblock \bibinfo{journal}{Mon Wea Rev} \bibinfo{volume}{133},
  \bibinfo{pages}{97--119}.

\end{thebibliography}

\end{document}